\documentclass[11pt]{amsart}
\usepackage[centertags]{amsmath}
\usepackage{amsfonts,amssymb}
\usepackage{graphicx}
\usepackage{enumerate}

  \newtheorem{theorem}{Theorem}
  
  \newtheorem{corollary}{Corollary}
  \newtheorem{proposition}{Proposition}
  \newtheorem{lemma}{Lemma}

\DeclareMathOperator{\re}{Re}

\begin{document}

\title[On the degrees of polynomial divisors over finite fields]{On the degrees of polynomial divisors \\ over finite fields}
\author{Andreas Weingartner}
\address{ 
Department of Mathematics,
351 West University Boulevard,
 Southern Utah University,
Cedar City, Utah 84720, USA}
\email{weingartner@suu.edu}
\date{\today}

\subjclass[2010]{11T06, 11N25, 05A05}

\begin{abstract}
We show that the proportion of polynomials of degree $n$ over the finite field with $q$ elements, which have a divisor of every degree below $n$, is given by $c_q n^{-1} + O(n^{-2})$.
More generally, we give an asymptotic formula for the proportion of polynomials, whose set of degrees of divisors has no gaps of size greater than $m$. To that end, we first derive an improved estimate for the proportion of polynomials of degree $n$, all of whose non-constant divisors have degree greater than $m$.
In the limit as $q \to \infty$, these results coincide with corresponding estimates related to the cycle structure of permutations.
\end{abstract} 
\maketitle

\section{Introduction}

There are many parallels between the factorization of integers into primes and the decomposition of combinatorial structures into components. 
For an overview with examples see the surveys \cite{ABT, GRA, Rud}. In this note we want to explore a correspondence between the distribution of integer divisors and the degree distribution of polynomial divisors over finite fields. 

Let $\mathbb{F}_q$ be the finite field with $q$ elements. What proportion of polynomials of degree $n$ over $\mathbb{F}_q$ have divisors of every degree below $n$? How does this question relate to the distribution of divisors of integers?

Let $F$ be a monic polynomial of degree $n$ over $\mathbb{F}_q$ with factorization $F=\prod_{1\le i \le j} P_i$, where the $P_i$ are irreducible monic polynomials. The set of degrees of divisors of $F$ is given by
$$ A_1= A_1(F)= \Biggl\{ \sum_{1\le i \le j} \varepsilon_i \deg(P_i) : \varepsilon_i \in \{0,1\} \Biggr\} \subseteq [0,n].$$
Hence $F$ has a divisor of every degree below $n$ if and only if $A_1=[0,n]\cap \mathbb{Z}$, i.e. the set $A_1$ has no gaps of size greater than $1$.

For an integer $N$ with prime factorization $N=\prod_{1\le i \le j} p_i $,  the set
$$A_2=A_2(N) = \Biggl\{ \sum_{1\le i \le j} \varepsilon_i \log p_i : \varepsilon_i \in \{0,1\}  \Biggr\}\subseteq [0,\log N] $$
has no gaps of size greater than $1$ if and only if every interval 
of the form $[x,ex)$, $1\le x\le N$, contains a divisor of $N$, i.e. the sequence of divisors of $N$,
$1=d_1<d_2<\ldots <d_k=N$, satisfies 
\begin{equation}\label{maxratio}
\max_{1\le i<k} \frac{d_{i+1}}{d_{i}} \le e.
\end{equation}

The similarity between the sets $A_1$ and $A_2$, with the logarithms of the integers ($\log N$ and $\log p_i$) taking the roles of the degrees of the polynomials ($n$ and  $\deg(P_i)$), suggests that the study of polynomials having divisors of every degree is related to the study of integers whose sequence of divisors satisfies \eqref{maxratio}. 
Improving on earlier estimates by Tenenbaum \cite{Ten86, Ten95} and Saias \cite{Sai}, we found \cite[Corollary 1.1]{PDD} that the number of such integers $\le x$ is given by 
$$ \frac{c_0 x}{\log x}\left(1+ O\left(\frac{1}{\log x}\right)\right),$$
for some positive constant $c_0$. Here we use a similar strategy to establish the analogous result for polynomials.

\begin{theorem}\label{thm1}
The proportion of polynomials of degree $n$ over $\mathbb{F}_q$, which have a divisor of every degree below $n$, is given by
$$
\frac{c_q}{n} \left(1+ O\left(\frac{1}{n}\right)\right).
$$
The factor $c_q$ depends only on $q$ and satisfies 
$$0< c_q  = C+ O\left(q^{-2\tau}\right),$$
where $C=(1-e^{-\gamma})^{-1}=2.280291...$, $\gamma$  denotes Euler's constant and the constant
$\tau=0.205466...$ is defined in Proposition \ref{etaqm}.
\end{theorem}

In Theorem \ref{thm1} and below, the implied constants in the error terms are absolute. 
In particular, all of our estimates are valid uniformly in $q$.

The method in \cite{PDD} relies on a sharp estimate by Tenenbaum \cite[Corollary 7.6, Section III.6]{Ten} for the number of integers $\le x$ which are free of prime divisors $\le y$,  for the entire range of $y$-values $2\le y \le x$.
The corresponding results for polynomials, which are available in the literature, are not precise enough for our method to succeed. 
Thus we first derive an improved estimate for $r(n,m)$, the proportion of polynomials of degree $n$ over $\mathbb{F}_q$, all of whose non-constant divisors have degree $>m$.  Such polynomials are sometimes referred to as $m$-rough, and estimates of $r(n,m)$ play a role in the analysis of factorization algorithms \cite{GP, PR}. Hence Theorem \ref{thm2} may be of independent interest.

We will need the following notation. 
Buchstab's function $\omega(u)$ is the 
unique continuous solution to the equation
\begin{equation*}
(u\omega(u))' = \omega(u-1) \qquad (u>2)
\end{equation*}
with initial condition
$\omega(u)=1/u$ for $1\le u \le 2$. Let $\omega(u)=0$ for $u<1$.
The number of monic irreducible polynomials of degree $n$ over $\mathbb{F}_q$ is given by \cite[Theorem 3.25]{LN} 
\begin{equation*}
 I_n=\frac{1}{n} \sum_{d|n} \mu \left(\frac{n}{d}\right) q^d  \qquad (n\ge 1),
\end{equation*}
where $\mu$ is the M\"obius function. An easy exercise \cite[p. 142, Ex. 3.26 and Ex. 3.27]{LN} shows that
\begin{equation}\label{In}
 \frac{q^n}{n} - \frac{2q^{n/2}}{n}< I_n \le \frac{q^n}{n} \qquad (n\ge 1) . 
\end{equation}
Define
$$\lambda_q(m) := \prod_{k=1}^m \left(1-\frac{1}{q^k}\right)^{I_k} .$$

\begin{theorem}\label{thm2}
Let $u=n/m$. For $n>m\ge 1$ we have
$$ r(n,m)=\lambda_q(m) \, e^\gamma \, \omega(u) \left(1+ O\left(\frac{(u/e)^{-u}}{m}\right)\right),$$
and
$$ \lambda_q(m)=e^{-H_m}  \left(1+O\left(\frac{1}{m q^{(m+1)/2}}\right)\right)
= \frac{e^{-\gamma}}{m} \left(1+O\left(\frac{1}{m}\right)\right),$$
where $\displaystyle H_m:=\sum_{1\le k \le m} \frac{1}{m}$.
\end{theorem}

With the second estimate for $\lambda_q(m)$, Theorem \ref{thm2} simplifies to 
\begin{corollary}\label{cora}
For $n>m\ge 1$ we have
$$r(n,m)= \frac{\omega(u)}{m} \left(1+O\left(\frac{1}{m}\right)\right).$$
\end{corollary}
Warlimont \cite[Eqs. (3) and (4)]{War} showed that Corollary \ref{cora} and the second estimate for $\lambda_q(m)$ hold in the more general setting of arithmetical semigroups.

Inserting the estimate for $\omega(u)$ from Lemma \ref{lemom}
into Theorem \ref{thm2}, we obtain the following improvement of \cite[Theorem 3.1]{GP} and \cite[Theorem 3.1]{PR}.
\begin{corollary}\label{corb}
For $n>m\ge 1$ we have
$$r(n,m)=\lambda_q(m)  \Bigl(1+O \bigl((u/e)^{-u}\bigr)\Bigr) .$$
\end{corollary}

When $m\ge 3$, we can replace $(u/e)^{-u}$ by $u^{-u}$ in the error terms of Theorem \ref{thm2} and Corollary \ref{corb}.

For $m \le n/\log n$, we derive Theorem \ref{thm2} by applying the residue theorem to the generating function of $r(n,m)$. 
When $m>n/\log n$, we show that $r(n,m)$ is very close to $p(n,m)$, the proportion of permutations of $n$ objects which 
have no cycles of length $\le m$. The result then follows from a recent estimate of $p(n,m)$ due to Manstavi\v{c}ius and Petuchovas \cite{Man}.  

With very little extra effort, we can generalize Theorem \ref{thm1} as follows.
Let $f(n,m)=f_q(n,m)$ be the proportion of monic polynomials $F$ of degree $n$ over  $\mathbb{F}_q$, with the property that the set of degrees of divisors of $F$ (i.e. the set $A_1$) has no gaps of size greater than $m$.  
The original question corresponds to $m=1$. The corresponding generalization in the case of the divisors of integers would be to replace the upper bound $e$ in \eqref{maxratio} by some parameter $t$. The asymptotic behavior of the number of integers up to $x$, all of whose ratios of consecutive divisors are at most $t$, is described in \cite[Theorem 1.3]{PDD}
in terms of a function $d(u)$, which is defined as follows. 

Let $d(u)=0$ for $u<0$ and 
\begin{equation}\label{dinteq}
d(u)= 1-\int_0^{\frac{u-1}{2}} \frac{d(v)}{v+1} \ \omega\left(\frac{u-v}{v+1}\right) \mathrm{d} v \qquad (u\ge  0),
\end{equation}
where $\omega(u)$ is Buchstab's function.
In \cite[Theorem 1]{IDD3}, we used equation \eqref{dinteq} to show that
\begin{equation}\label{IDD3T1}
d(u)=\frac{C}{u+1}\, \Bigl(1+ O\left(u^{-2}\right)\Bigr) \qquad (u \ge 1),
\end{equation}
where  $C=(1-e^{-\gamma})^{-1}=2.280291...$,  as in Theorem \ref{thm1}.

\begin{figure}[htb]\label{fig1}
\begin{center}
\includegraphics[height=5.2cm,width=11cm]{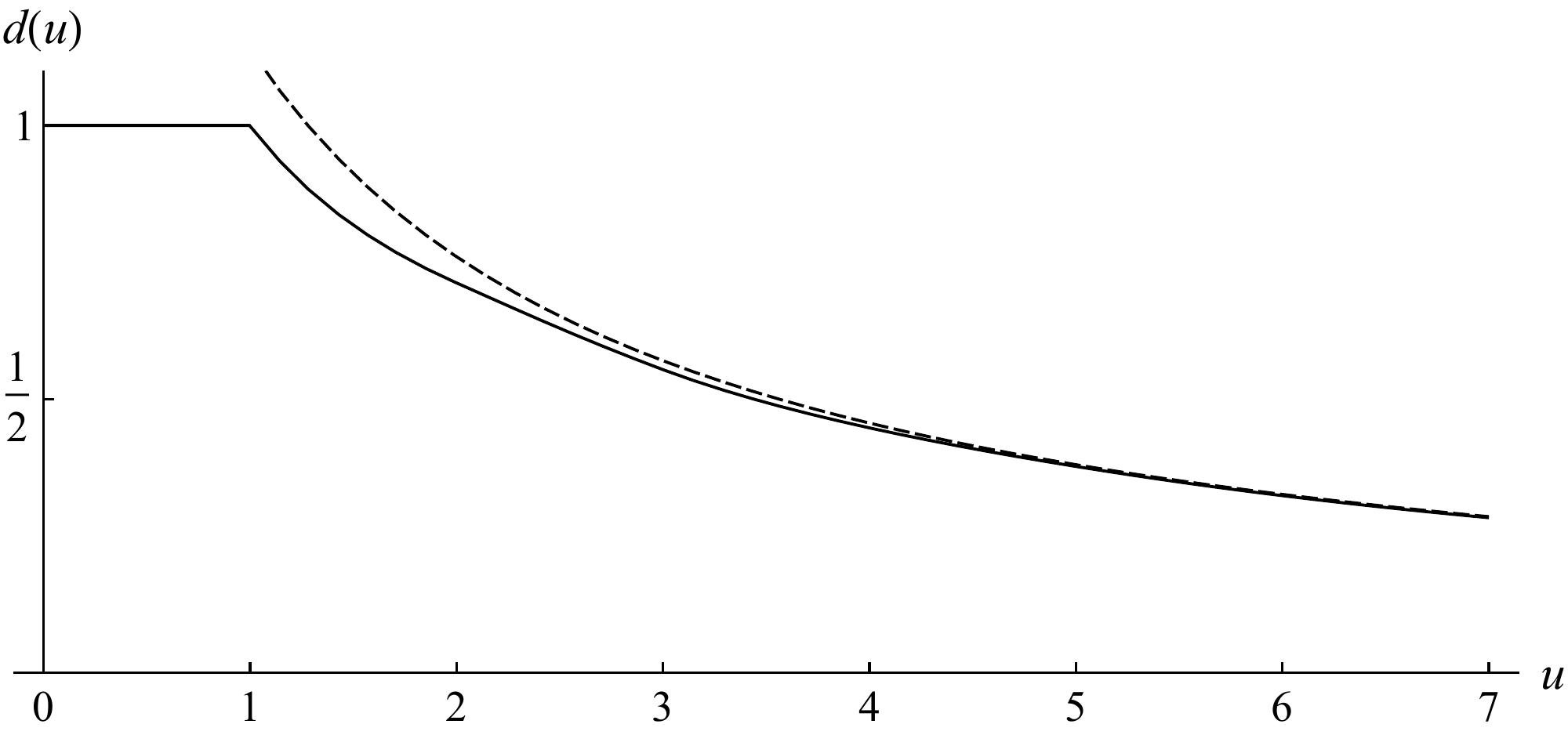}
\caption{The graphs of $d(u)$ (solid) and $\frac{C}{u+1}$ (dashed).}
\end{center}
\end{figure}

The following result is the polynomial analogue of \cite[Theorem 1.3]{PDD}.
\begin{theorem}\label{thm3}
For $n\ge 0$, $m\ge 1$ we have
$$ f(n,m)= \eta_q(m) \, d(n/m) \left( 1+O\left(\frac{1}{n+m}\right)\right) .$$
The factor $\eta_q(m)$ depends only on $q$ and $m$, and satisfies
$$ 0<  \eta_q(m) = 1 + O\left(\frac{1}{m q^{(m+1)\tau}}\right),$$ 
where $\tau =0.205466...$ is defined in Proposition \ref{etaqm}.
\end{theorem}

Depending on the size of $m$ relative to $n$, we can simplify Theorem \ref{thm3} in different ways. 
With the estimate \eqref{IDD3T1}, we get 
\begin{corollary}\label{cor1}
For $n\ge m\ge 1$ we have
$$f(n,m)=  \frac{c_q(m) m}{n+m} \left( 1+O\left(\frac{m^2}{n^2}+\frac{1}{n}\right)\right),$$
where 
$$ 0< c_q(m):= C \eta_q(m)=C + O\left(\frac{1}{m q^{ (m+1)\tau }}\right).$$ 
\end{corollary}
Theorem \ref{thm1} follows from Corollary \ref{cor1} with $m=1$ and $c_q=c_q(1)$.
Corollary \ref{cor1} clearly implies
\begin{corollary}\label{cor2}
For $n\ge m\ge 1$ we have
$$f(n,m)= \frac{C m}{n+m} \left(1+O\left(\frac{m^2}{n^2}+\frac{1}{n}+\frac{1}{m q^{(m+1)\tau}}\right)\right).$$
\end{corollary}
Finally, the estimate for $\eta_q(m)$ in Theorem \ref{thm3} implies
\begin{corollary}\label{cor3}
For $n\ge 0$, $m\ge 1$ we have
$$ f(n,m)=  d(n/m) \left(1+O\left(\frac{1}{n+m}+\frac{1}{m q^{(m+1)\tau}}\right)\right).$$
\end{corollary}

We now turn to analogous results in the context of permutations and their decomposition into disjoint cycles. 
Let $S_n$ be the symmetric group on $\{1,2, \ldots ,n\}$.  
One reason for considering permutations is that our proof of the asymptotic result for the factor $\eta_q(m)$ in Theorem \ref{thm3} (and hence for the factor $c_q$ in Theorem \ref{thm1}) depends on Theorem \ref{thm5}, the corresponding result in the context of permutations. 
Moreover, each of the following estimates related to permutations coincides with the limit, as $q\to \infty$, of the corresponding estimate related to polynomials over $\mathbb{F}_q$. Theorem \ref{thm1} takes the following form. 

\begin{theorem}\label{thm4}
The proportion of permutations $\sigma \in S_n$, which have a cycle structure such that every positive integer below $ n$ can be expressed as a sum of lengths of distinct cycles of $\sigma$, is given by
$$
\frac{C}{n}\left( 1 + O\left(\frac{1}{n}\right)\right),
$$
where $C$ is as in Theorem \ref{thm1}.
\end{theorem}

An alternate way to state Theorem \ref{thm4} is as follows. The proportion of permutations $\sigma \in S_n$, with the property that for every positive integer $m\le n$ there exists a set $M \subseteq \{1,2,3,\ldots,n\}$ with cardinality $m$ such that $\sigma(M)=M$, is given by $Cn^{-1}+O(n^{-2})$.

The analogue of Theorem \ref{thm2} in the context of permutations is Proposition \ref{fullp}, which depends on a recent result by Manstavi\v{c}ius and Petuchovas \cite{Man}. 
The analogue of Theorem \ref{thm3} also holds in this context. Assume a permutation $\sigma \in S_n$ has cycles of length $l_1, l_2, \ldots, l_j$. 
The set
$$ A_3=A_3(\sigma)= \Biggl\{ \sum_{1\le i \le j} \varepsilon_i l_i  : \varepsilon_i \in \{0,1\} \Biggr\} \subseteq [0,n]$$ 
represents the set of cardinalities of sets $M \subseteq \{1,2,3,\ldots,n\}$ such that $\sigma(M)=M $.
Let $g(n,m)$ denote the proportion of permutations in $S_n$ with the property that $A_3$ has no gaps of size greater than $m$. 

\begin{theorem}\label{thm5}
For $n\ge 0$, $m\ge 1$ we have
$$ g(n,m)= d(n/m) \left( 1+O\left(\frac{1}{n+m}\right)\right)  .$$
\end{theorem}

Unlike Theorem \ref{thm3} and \cite[Theorem 1.3]{PDD}, 
Theorem \ref{thm5} does not involve a factor $\eta(m)$, which suggests that
this variant of the problem, dealing with permutations, exhibits the simplest behavior.
With the estimate \eqref{IDD3T1}, we get 
\begin{corollary}\label{cor1P}
For $n\ge m\ge 1$ we have
$$g(n,m)=  \frac{C m}{n+m} \left( 1+O\left(\frac{m^2}{n^2}+\frac{1}{n}\right)\right).$$
\end{corollary}
Theorem \ref{thm4} follows from Corollary \ref{cor1P} with $m=1$.

\medskip

The situation is somewhat different in the context of integer partitions.
Erd\H{o}s and Szalay \cite{ES}, and later Dixmier and Nicolas \cite{DN}, investigated the problem analogous to Theorems \ref{thm1} and \ref{thm4}. For the integer partition $n=l_1+l_2+ \ldots + l_j$, consider the set
$$ A_4 =  \Biggl\{\sum_{1\le i \le j} \varepsilon_i l_i  : \varepsilon_i \in \{0,1\} \Biggr\} \subseteq [0,n].$$ 
Dixmier and Nicolas \cite{DN} call a partition \emph{practical} if $A_4=[0,n]\cap \mathbb{Z}$. They show that the proportion of partitions of $n$ which are practical is given by (see \cite[Theorem 2]{DN} for the full  result) 
$$ 1-\frac{\pi}{\sqrt{6n}} + O\left(\frac{1}{n}\right).$$
Thus, unlike with integers, polynomials and permutations, almost all partitions satisfy $A_4=[0,n]\cap \mathbb{Z}$.

A related problem is the study of the sequence of degrees $n$ for which the particular polynomial $X^n-1$ has divisors of every degree in $K[X]$, where $K$ is a given field. 
By adapting the work of Saias \cite{Sai} on practical numbers, Thompson \cite{T1} found that the number of such $n\le x$ is $\asymp x/\log x$ when $K=\mathbb{Q}$.  This was extended to any number field $K$ by Pollack and Thompson \cite{PT}. When $K=\mathbb{F}_p$, the true order of magnitude of the number of such $n\le x$ is still unknown, 
but Thompson \cite{T2} shows that they have asymptotic density zero by assuming the Generalized Riemann Hypothesis.

\subsection*{Acknowledgements}
The author thanks Paul Pollack for suggesting the question which is answered in Theorem \ref{thm1}, and Eugenijus
Manstavi\v{c}ius for providing some references.
\section{Proof of Theorem \ref{thm2}}

 When $m\le n/\log n $, the first statement in Theorem \ref{thm2} follows from Lemma \ref{lemom} and Proposition \ref{rsm}.
When $n/\log n < m < n$, Theorem \ref{thm2} follows from Propositions \ref{Mansta}, \ref{rap}  and  \ref{cqh}. The second statement in Theorem \ref{thm2} is Proposition \ref{cqh}.
We put $r(0,m)=1$ for $m\ge 0$. 

\begin{lemma}\label{lemom}
For $u>1$, we have $\omega(u)- e^{-\gamma} \ll \exp(-u\log (u\log u) +O(u)) .$
\end{lemma}

\begin{proof}
This follows from  \cite[Theorem III.6.4]{Ten}.
\end{proof}

\begin{proposition}\label{rsm}
Let  $u=n/m$. For $1\le m \le n/\log n$, 
$$r(n,m)=\lambda_q(m) \Bigl(1+O \bigl((u/e)^{-u} m^{-1}\bigr)\Bigr).$$
If $m\ge 3$, we can replace $e$ by $1$ in the error term.
\end{proposition}

\begin{proof}
The number of monic polynomials of degree $n$, which have no non-constant divisors of degree $\le m$, is given by $[z^n] F_m(z)$, the coefficient of $z^n$ in the power series of 
$$ F_m(z):= \prod_{k>m} \left(1-z^k\right)^{-I_k} = \frac{1}{1-qz} \prod_{1\le k\le m} \left(1-z^k\right)^{I_k}
\quad (|z|<1/q).$$
Cauchy's residue theorem yields
$$r(n,m)= \frac{[z^n] F_m(z)}{q^n} = [z^n] F_m(z/q) 
= \frac{1}{2 \pi i} \int_{|z|=1/2} F_m(z/q) \frac{\mathrm{d} z}{z^{n+1}}.$$
Stretching the contour to $|z|=R>1$ leaves a residue of $\lambda_q(m)$ from the pole at $z=1$. 
When $|z|=R$ we have
\begin{equation*}
\bigl|F_m(z/q)\bigr| \le \frac{1}{R-1}\prod_{1\le k \le m } \left(1+\frac{R^k}{q^k}\right)^{I_k}
\le \frac{1}{R-1} \exp\left(\sum_{1\le k \le m} \frac{R^k }{k} \right),
\end{equation*}
since  $I_k \le q^k/k$ by \eqref{In} and $\log(1+x)\le x$ for $x>0$. Hence
\begin{equation}\label{rcd}
\bigl|r(n,m)-\lambda_q(m)\bigr| \le \frac{1}{R-1}  \exp\left(\sum_{1\le k \le m} \frac{R^k }{k} \right) \frac{1}{R^{n}}.
\end{equation}
If $m=1$ or $2$, the choice $R^m=n$ makes  the right-hand side of \eqref{rcd} $\ll (u/e)^{-u} m^{-2}$. 
For $m\ge 3$, we write 
\begin{equation}\label{sumzj}
\begin{split}
\sum_{1\le k \le m} \frac{R^k }{k} & = H_m +  \sum_{1\le k \le m} \frac{R^k-1 }{k} = H_m +  \sum_{1\le k \le m} \int_0^{\log R} e^{kt}  \, \mathrm{d} t \\
& =H_m+ \int_0^{\log R}  (e^t -1+1) \, \frac{e^{mt}-1}{e^t-1}  \, \mathrm{d} t \\
&= H_m+ \int_0^{\log R} (e^{mt}-1)  \, \mathrm{d}t + \int_0^{\log R}  \frac{e^{mt}-1}{e^t-1}  \, \mathrm{d}t \\
& \le H_m + \frac{R^m}{m} +  \int_0^{\log R^m}  \frac{e^t-1}{t}  \, \mathrm{d}t,
\end{split}
\end{equation}
since $e^t-1\ge t$.
A little calculus exercise shows that for $x>0$ we have
$$ \int_0^x \frac{e^t -1}{t}  \, \mathrm{d}t \le \frac{e^x}{x}\left(1+\frac{2}{x}\right).$$
Hence
\begin{equation}\label{ser}
\sum_{1\le k \le m} \frac{R^k }{k}  \le H_m + \frac{R^m}{m}  
+ \frac{R^m}{\log R^m} \left(1+\frac{2}{\log R^m}\right).
\end{equation}
When $3\le m \le \log n / \log\log\log n$, we choose $R$ such that $R^m=n$. 
If $\log n / \log\log\log n <m\le n/\log n$, we choose $R$ such that $R^m=u \log u$. In each case, \eqref{ser} shows that the right-hand side of \eqref{rcd} is $\ll u^{-u} m^{-2}$. 
 \end{proof}

\begin{lemma}\label{fer}
For $n>m\ge 1$,
$$ n \, r(n,m)=1+\sum_{m<k<n-m} r(k,m) +  O\left( \frac{1}{q^{n/2}} + \frac{1}{m q^{(m+1)/2}} \right).$$
\end{lemma}

\begin{proof}
Let $R(n,m)=q^n r(n,m)$, the number of monic polynomials of degree $n$, which have no non-constant divisors of degree $\le m$. 
As in the proof of Proposition \ref{rsm}, the generating function of $R(n,m)$ is given by 
$$ \sum_{n\ge 0} R(n,m) z^n = F_m(z) = \prod_{k>m} \left(1-z^k\right)^{-I_k} = \exp\left( -\sum_{k>m} I_k \log(1-z^k) \right).$$
Differentiating with respect to $z$ gives
$$  \sum_{n\ge 1} n\, R(n,m) z^{n-1}=F'_m(z) = F_m(z) \sum_{k>m} I_k \frac{k z^{k-1}}{1-z^k} $$
and hence
$$  \sum_{n\ge 1} n\, R(n,m) z^{n-1} = \left( \sum_{n\ge 0} R(n,m) z^n\right) \left(\sum_{k>m} k I_k   \sum_{j\ge 1} z^{kj-1}\right).$$
Comparing coefficients of $z^{n-1}$, we find that
$$  n\, R(n,m) = \sum_{k>m} k I_k   \sum_{j\ge 1} R(n-kj,m)  .$$
We estimate $k I_k$ by \eqref{In} and divide both sides by $q^n$ to get
\begin{equation}\label{rer}
  n \, r(n,m) =  \sum_{k>m} r(n-k,m) + O\left(   \sum_{m<k\le n} q^{-k/2} r(n-k,m) + \sum_{j\ge 2} \sum_{k>m} q^{-k(j-1)} \right).
\end{equation}
Note that
$$  \sum_{k>m} r(n-k,m)  = 1+\sum_{m<k<n-m} r(k,m),$$
since $r(0,m)=1$ and $r(k,m)=0$ for $1\le k \le m$. The double sum in the error term of \eqref{rer} is $\ll q^{-(m+1)}$, since $q\ge 2$. 
We have $r(n-k,m)=1$ when $k=n$ and $r(n-k,m) \ll 1/m$ when $k<n$. Hence the first sum in the error term of \eqref{rer} is acceptable.
\end{proof}

Recall that $p(n,m)$ denotes the proportion of permutations of $n$ objects having no cycles of length $\le m$. We put $p(0,m)=1$. 

\begin{proposition}\label{psm}
Let $u=n/m$. For $1\le m \le n/\log n$, 
$$p(n,m)=e^{-H_m} \Bigl(1+O \bigl((u/e)^{-u} m^{-1}\bigr)\Bigr).$$
 If $m\ge 3$, we can replace $e$ by $1$ in the error term.
\end{proposition}

\begin{proof}
The generating function for $p(n,m)$ satisfies \cite[Ex. IV.9]{FS}
\begin{equation}\label{Dgen}
D_m(z) := \sum_{n\ge 0} p(n,m)z^n = \exp \left(\sum_{k\ge m+1} \frac{z^k}{k} \right) 
= \frac{1}{1-z} \exp \left(-\sum_{1\le k\le m} \frac{z^k}{k} \right)  ,
\end{equation}
for $|z|<1$. 
Cauchy's residue theorem yields
$$p(n,m)= \frac{1}{2 \pi i} \int_{|z|=1/2} D_m(z) \frac{\mathrm{d}z}{z^{n+1}}.$$
Stretching the contour to $|z|=R>1$ leaves a residue of $e^{-H_m}$ from the pole at $z=1$. 
When $|z|=R$ we have
\begin{equation*}
\bigl|D(z)\bigr| \le \frac{1}{R-1} \exp\left(\sum_{1\le k \le m} \frac{R^k }{k} \right),
\end{equation*}
hence
\begin{equation*}
\bigl|p(n,m)-e^{-H_m}\bigr| \le \frac{1}{R-1}  \exp\left(\sum_{1\le k \le m} \frac{R^k }{k} \right) \frac{1}{R^{n}}.
\end{equation*}
The remainder of the proof is identical to that of Proposition \ref{rsm} following the estimate \eqref{rcd}.
\end{proof}

The following result is a somewhat weaker version of a recent result by Manstavi\v{c}ius and Petuchovas \cite[Theorem 1.1]{Man}.
\begin{proposition}[Manstavi\v{c}ius and Petuchovas]\label{Mansta}
Let $u=n/m$. For $n>m\ge \sqrt{n \log n}$ we have
$$ p(n,m) = e^{\gamma-H_m} \omega(u) \left(1+O\left(\frac{u^{-u}}{m}\right)\right).$$
\end{proposition}

Combining Propositions \ref{psm} and \ref{Mansta} with Lemma \ref{lemom}, we get the following estimate.
\begin{proposition}\label{fullp}
Let $u=n/m$. For $n>m\ge 1$ we have
$$ p(n,m) = e^{\gamma-H_m} \omega(u) \left(1+O\left(\frac{(u/e)^{-u}}{m}\right)\right).$$
\end{proposition}

The next estimate is an immediate consequence of Manstavi\v{c}ius \cite[Theorem 3]{Man02}.
It is also follows from Propostion \ref{fullp}.
\begin{proposition}[Manstavi\v{c}ius]\label{pub}
For $n>m\ge 1$, we have 
$$ p(n,m) = \frac{\omega(n/m)}{m}\left(1 +O\left(\frac{1}{m}\right)\right).$$
\end{proposition}

\begin{lemma}\label{fep}
For $n>m\ge 0$,
$$ n \, p(n,m)=1+\sum_{m<k<n-m} p(k,m) .$$
\end{lemma}

\begin{proof}
From \eqref{Dgen} we have
$$ D'_m(z) = \sum_{n\ge 1} n \, p(n,m)z^{n-1} = D_m(z) \sum_{k\ge m} z^k.$$
Comparing the coefficients of $z^{n-1}$ in the last equation leads to
$$ n\, p(n,m) = \sum_{m\le k \le n-1} p(n-1-k,m) = 1 + \sum_{m<k<n-m} p(k,m) ,$$
for $n> m\ge 0$, since $p(0,m)=1$ and $p(k,m)=0$ for $1\le k \le m$.
\end{proof}

\begin{proposition}\label{rap}
For $n\ge 1, m\ge 1$, 
$$r(n,m)=p(n,m) + O\left( \frac{1}{n q^{n/2}} + \frac{1}{m^2 q^{(m+1)/2}} \right).$$
\end{proposition}

\begin{proof}
If $m\ge n \ge 1$, then $r(n,m)=p(n,m)=0$. Hence we may assume $n>m\ge 1$.
Let $s(n,m)=r(n,m)-p(n,m).$ Lemmas \ref{fer} and \ref{fep} imply
\begin{equation}\label{qrec}
  |s(n,m)| \le \frac{B}{nq^{n/2}}+ \frac{B}{nm q^{(m+1)/2}} + \frac{1}{n} \sum_{m<k<n-m} |s(k,m)|, 
\end{equation}
where $B$ is the implied constant in the error term of Lemma \ref{fer}. 
When $n/2 \le m <n$, the last sum is empty and we have
\begin{equation}\label{ind}
 |s(n,m)|  \le  \frac{B}{nq^{n/2}} + \frac{4B}{m^2 q^{(m+1)/2}}.
\end{equation}
Now assume that \eqref{ind} holds for all $n\ge 1$ and $n/j \le m <n$, for some $j\ge 2$.
Let $n/(j+1) \le m < n$. Then $n\le m(j+1)$ and $k<n-m$ imply $k<mj$. Hence \eqref{qrec} 
and the inductive hypothesis show that 
\begin{equation*}
\begin{split}
|s(n,m)| & \le  \frac{B}{nq^{n/2}} + \frac{B}{nm q^{(m+1)/2}} 
+ \frac{1}{n} \sum_{m<k<n-m}  \left(\frac{B}{k q^{k/2}} + \frac{4B}{m^2 q^{(m+1)/2}}\right) \\
& \le  \frac{B}{nq^{n/2}}+ \frac{4B}{m^2 q^{(m+1)/2}} \left(\frac{m}{4n}+  \frac{m}{4n}\sum_{k\ge 0} q^{-k/2} + \frac{n-2m}{n}\right) \\
&  \le \frac{B}{nq^{n/2}}+ \frac{4B}{m^2 q^{(m+1)/2}},
\end{split}
\end{equation*}
since $\sum_{k\ge 0} q^{-k/2} \le \sum_{k\ge 0} 2^{-k/2} < 4$.  
Thus \eqref{ind} holds for all $n\ge 1$ and all $m$ with $0<m<n$.
\end{proof}

\begin{proposition}\label{cqh}
For $m\ge 1$ we have
$$ \lambda_q(m)=e^{-H_m}  \left(1+O\left(\frac{1}{m q^{(m+1)/2}}\right)\right)
= \frac{e^{-\gamma}}{m} \left(1+O\left(\frac{1}{m}\right)\right).$$
\end{proposition}

\begin{proof}
Proposition \ref{rsm} implies that 
$ \lim_{n\to \infty} r(n,m) = \lambda_q(m)$, 
while Proposition \ref{psm} (or \cite[Ex. IV.9]{FS}) shows that $\lim_{n\to \infty} p(n,m) = e^{-H_m}$.
Letting $n\to \infty $ in Proposition \ref{rap} gives the first equation. 
The second equation follows from the well-known estimate $H_m= \log m + \gamma + O(m^{-1})$.
\end{proof}

\section{Proof of Theorem \ref{thm3}}

Let $\mathcal{A}(n,m)$ denote the set of monic polynomials $F$ of degree $n$ over $\mathbb{F}_q$ with the property that the set of degrees of divisors of $F$ has no gaps of size greater than $m$. We have $f(n,m)= q^{-n} |\mathcal{A}(n,m)| $. 

\begin{lemma}\label{cond}
A monic polynomial $F$ of degree $n$, 
$F=P_1 P_2 \cdots P_j$, where the $P_i$ are monic irreducible polynomials with $\deg(P_1) \le \deg(P_2) \le \ldots \le \deg(P_j)$,
satisfies $F \in \mathcal{A}(n,m)$ if and only if 
\begin{equation}\label{conde}
\deg(P_i) \le \deg(P_1\cdots P_{i-1}) + m \qquad (1\le i \le j).
\end{equation}
\end{lemma}

\begin{proof}
If $\deg(P_i) > \deg(P_1\cdots P_{i-1}) + m$ for some $1\le i \le j$, then the open interval $ ( \deg(P_1\cdots P_{i-1}) , \deg(P_i))$ is of length $>m$ and contains no degrees of divisors of $F$. Thus \eqref{conde} is a necessary condition for $F \in \mathcal{A}(n,m)$. To see that  \eqref{conde} is sufficient, we use induction on $j$. The case $j=1$ is obvious. Assume that \eqref{conde} implies $F \in \mathcal{A}(n,m)$ for some $j\ge 1$. Let
$ F=P_1 P_2 \cdots P_j P_{j+1}$ satisfy \eqref{conde}, with $j$ replaced by $j+1$. The set of degrees of divisors of $F$ is $D \cup (D+\deg(P_{j+1}))$, where $D$ is the set of degrees of divisors of $P_1 P_2 \cdots P_j $. Now $D$ has no gaps greater than $m$ by the inductive hypothesis and $\deg(P_{j+1})\le m+ \max D $ by \eqref{conde}. Thus  $D \cup (D+\deg(P_{j+1}))$ also has no gaps greater than $m$ and $F \in \mathcal{A}(n,m)$.
\end{proof}

\begin{lemma}\label{fe}
For $n\ge 0$, $m\ge 0$ we have
$$ 1=\sum_{0\le k \le n} f(k,m) \, r(n-k,k+m)  .$$
\end{lemma}

\begin{proof}
Lemma \ref{cond} implies that each monic polynomial $F$ of degree $n$ factors uniquely as $F=QR$, where $F=P_1 P_2 \cdots P_j$,  the $P_i$ are monic irreducible polynomials with $\deg(P_1) \le \deg(P_2) \le \ldots \le \deg(P_j)$, 
$$ Q=P_1\cdots P_{i_0} , \quad R=P_{i_0+1}\cdots P_j ,$$
and $i_0$ is the unique index such that
$$ Q \in \mathcal{A}(\deg(Q),m), \quad \deg{P_{i_0+1}}>\deg(Q)+m.$$
The case $i_0=0$ corresponds to $Q=1$ and $R=F$, while $i_0=j$ means $Q=F\in  \mathcal{A}(n,m)$ and $R=1$.
Given $\deg(Q)=k$, there are $q^{k} f(k,m)$ choices for $Q$ and $q^{n-k}r(n-k,k+m)$ choices for $R$. Counting all of the $q^n$ monic polynomials of degree $n$ according to the degree of $Q$, we find that
$$ q^n = \sum_{0\le k \le n}q^{k}  f(k,m) \, q^{n-k}r(n-k,k+m),$$
from which the result follows.
\end{proof}

\begin{lemma}\label{fe2}
For $n\ge 0$, $m\ge 0$ we have
$$ 1= f(n,m)+\sum_{0\le k < \frac{n-m}{2}} f(k,m) \, r(n-k,k+m) .$$
\end{lemma}

\begin{proof}
This follows from Lemma \ref{fe} since $r(0,n+m)=1$ and $r(n-k,k+m)=0$ if $\frac{n-m}{2} \le k < n$. 
\end{proof}

\begin{lemma}\label{se}
For $m\ge 0$ we have
$$ 1=\sum_{k\ge 0} f(k,m) \, \lambda_q(k+m).$$
\end{lemma}

\begin{proof}
If $m=0$, we have $f(0,0)=1$, $\lambda_q(0)=1$ and $f(k,0)=0$ for $k\ge 1$. Now fix any $m\ge 1$. We have $r(n-k,k+m) \asymp (k+m)^{-1}$ for $n-k>k+m$, by Corollary \ref{cora}. Lemma \ref{fe2} implies that $\sum_{0 \le k < \frac{n-m}{2}} \frac{f(k,m)}{k+m}$ is bounded from above as $n\to \infty$ and hence  $\sum_{k\ge 0} \frac{f(k,m)}{k+m}$ converges. Thus
$$ \liminf_{k\to \infty} f(k,m)=0, \quad \lim_{n\to \infty} \sum_{n/\log n < k < \frac{n-m}{2}} f(k,m) \, r(n-k,k+m) = 0.$$
Lemma \ref{fe2} shows that, as $n\to \infty$,
\begin{equation*}
\begin{split}
1 &= o(1)+ f(n,m)  + \sum_{k\le n/\log n} f(k,m) \, r(n-k,k+m)  \\
& =  o(1)  + f(n,m)+ \sum_{k\le n/\log n} f(k,m) \, \lambda_q(k+m),
\end{split}
\end{equation*}
by Corollary \ref{corb}. Since $ \liminf_{n\to \infty} f(n,m)=0$ and the last sum is increasing in $n$, the result follows. 
\end{proof}

For real $t\ge 0$, we define $f(t,m)=f(\lfloor t \rfloor, m)$. 

\begin{lemma}\label{inteq}
For $n\ge 0$, $m \ge 1$ we have
$$ f(n,m) = \int_0^\infty \frac{f(t,m)}{t+m} \left( e^{-\gamma} - \omega\left(\frac{n+m}{t+m}-1\right) \right)   \mathrm{d}t + E(n,m), $$
where
$$ E(n,m) \ll  E_0(n,m):=\int_0^\infty \frac{f(t,m)}{(t+m)^2}\exp\left(-\frac{n+m}{t+m}\right)  \mathrm{d}t  + \frac{f(\frac{n-m}{2},m)}{n+m}. $$
\end{lemma}

\begin{proof}
Lemmas \ref{fe2} and \ref{se} imply
\begin{multline*}
 f(n,m)= \sum_{0\le k < \frac{n-m}{2}} f(k,m) \,\Bigl(\lambda_q(k+m)- r(n-k,k+m)\Bigr) \\
+\sum_{k \ge  \frac{n-m}{2}} f(k,m) \, \lambda_q(k+m) =: S_1 + S_2,
\end{multline*}
say. We approximate $r(n-k, k+m)$ in $S_1$ by Theorem \ref{thm2} to get
$$ S_1 =   \sum_{0\le k < \frac{n-m}{2}} f(k,m) \, \lambda_q(k+m) \, e^\gamma 
\left(e^{-\gamma} -\omega\left(\frac{n-k}{k+m}\right)\right) +E_1(n,m),$$
where $ E_1(n,m) \ll  E_0(n,m). $
Next, we approximate $\lambda_q(k+m)$ in $S_1$ and $S_2$ by the second estimate in Proposition \ref{cqh}.
The resulting error term is again $\ll E_0(n,m)$ by Lemma \ref{lemom}. Thus
$$  f(n,m)=S_1+S_2 = \sum_{k\ge 0} \frac{f(k,m)}{k+m} \left( e^{-\gamma} - \omega\left(\frac{n+m}{k+m}-1\right) \right)
+E_2(n,m),$$
where $ E_2(n,m) \ll E_0(n,m)$,
since $\omega(1)=1$ and $\omega(u)=0$ for $u<1$. 
It remains to replace the sum by an integral. If $g(k)$ represents the $k$-th term of the sum, we estimate the error simply by 
$g(k) - \int_k^{k+1} g(t)\,  \mathrm{d}t \ll \max_{t\in [k,k+1]} |g'(t)|$ and use Lemma \ref{lemom} and $\omega'(u)\ll e^{-2u}$ for $u\ge 1$ (see \cite[Theorem III.6.4]{Ten}). The discontinuity of $\omega(u)$ at $u=1$ leads to an error of $\ll  f(\frac{n-m}{2},m)/(n+m) $. Thus we find that the error from replacing the sum by the integral is also $\ll E_0(n,m)$.
\end{proof}

\begin{proof}[Proof of Theorem \ref{thm3}]
Note that replacing the integer $n$ in Lemma \ref{inteq} by the real number $x$ with $n=[x]$ leaves the left-hand-side unchanged, and alters the integral by $\ll E_0(x,m)$. Thus, for $x\ge 0$, 
$$ f(x,m) = \alpha(m)- \int_0^\infty   \frac{f(t,m)}{t+m}  \, \omega\left(\frac{x+m}{t+m}-1\right)    \mathrm{d}t + E(x,m), $$
where 
$$ \alpha(m) := e^{-\gamma}  \int_0^\infty \frac{f(t,m)}{t+m}  \, \mathrm{d}t$$
and $E(x,m) \ll E_0(x,m)$. Lemma \ref{se} ensures that the integral defining $\alpha(m)$ converges and that $\alpha(m)\ll 1$.
Let $z\ge 0$ be given by $e^z=x/m+1$ and $y\ge 0$ by $e^y=t/m+1$. Define $G_m(z):=f(m(e^z-1),m)\,  e^z$ and 
$ \Omega(z):=\omega(e^z-1)$. We get
$$ G_m(z)=\alpha(m) e^z - \int_0^z G_m(y) \, \Omega(z-y) \, e^{z-y}  \, \mathrm{d}y + E_m(z),$$
where 
\begin{equation}\label{error}
E_m(z) \ll \frac{e^z}{m} \int_0^\infty \frac{G_m(y)}{e^{2y}} \exp\left( -e^{z-y}\right)  \mathrm{d}y + \frac{G_m(z-\log 2)}{me^z}.
\end{equation}
The last equation allows us to calculate Laplace transforms:
$$ \widehat{G}_m(s) = \frac{\alpha(m)}{s-1} - \widehat{G}_m(s) \widehat{\Omega}(s-1) + \widehat{E}_m(s) \qquad (\re s >1). $$
Hence
$$\widehat{G}_m(s) = \frac{\alpha(m)}{(s-1) (1+\widehat{\Omega}(s-1))} + \frac{\widehat{E}_m(s)}{1+\widehat{\Omega}(s-1)}
\qquad (\re s >1).$$
Equation \eqref{dinteq} written in terms of 
$ G(z):= e^z d(e^z -1) \asymp 1 $
is
$$ G(z) =e^z - \int_{0}^{z} G(y) \, \Omega(z-y) e^{z-y} \mathrm{d}y,$$ 
which shows that the Laplace transform of $G(z)$ is given by
$$ \widehat{G}(s) = \frac{1}{(s-1) (1+\widehat{\Omega}(s-1))} \qquad (\re s >1).$$
Thus
\begin{equation*}
\begin{split}
\widehat{G}_m(s) & = \alpha(m) \widehat{G}(s) + \widehat{E}_m(s) \widehat{G}(s) (s-1) \\
& = \alpha(m) \widehat{G}(s) +\widehat{E}_m(s) ( \widehat{G'}(s) - \widehat{G}(s)+1 ),
\end{split}
\end{equation*}
since $G(0)=1$. Now
$$ G'(y)-G(y) = e^{2y} d'(e^y-1) = -C + O\left(e^{-2y}\right)$$
by \cite[Corollary 5]{IDD3}. Inversion of the Laplace transforms yields
\begin{equation}\label{rep}
G_m(z) = \alpha(m) G(z) + E_m(z)+  \int_0^z E_m(y)\left( -C + O\left(e^{-2(z-y)}\right)\right) \mathrm{d}y .
\end{equation}
Since $0\le f(x,m) \le 1$, we have $0\le G_m(z) \le e^z$. Thus \eqref{error} shows that $E_m(z) \ll 1$ and \eqref{rep} yields $G_m(z) \ll 1+z$.  Using this estimate in \eqref{error} we find that $E_m(z)\ll (1+z) e^{-z}$ and \eqref{rep} now gives $G_m(z) \ll 1$. 
Appealing to \eqref{error} one last time, we finally arrive at 
$$E_m(z) \ll \frac{1}{m e^{z}}.$$
Hence
\begin{equation*}\label{betadef}
\beta(m) := -\int_0^\infty E_m(y) \, \mathrm{d}y 
= -\int_0^z E_m(y)\, \mathrm{d}y + O\left(\frac{1}{m e^z }\right)
\end{equation*}
and
$$ \int_0^z | E_m(y)| \, e^{-2(z-y)}\, \mathrm{d}y \ll  \frac{1}{m e^z }.$$
Thus \eqref{rep} yields
\begin{equation*}
G_m(z) = \alpha(m) G(z) + C \beta(m) +O\left(\frac{1}{m e^z }\right),
\end{equation*}
that is
\begin{equation*}\label{last}
f(x,m)= \alpha(m) \, d(x/m) + \frac{C \beta(m)}{x/m+1} +O\left(\frac{m}{(x+m)^2}\right),
\end{equation*}
for $x\ge 0$,  $m \ge 1$. Letting $x=m$ in the last equation shows that $\alpha(m)=1+O(1/m)$, since $\beta(m)\ll 1/m$, $f(m,m)=1$ and $d(1)=1$. 
With 
\begin{equation}\label{etadef}
\eta_q(m):= \alpha(m)+\beta(m)=1+O(1/m), 
\end{equation}
the estimate \eqref{IDD3T1} yields
\begin{equation}\label{lastf}
f(x,m)= \eta_q(m) \, d(x/m) +O\left(\frac{m}{(x+m)^2}\right).
\end{equation}
We need to show that $\eta_q(m) >0 $.
If $\eta_q(m)=0$ for some $m\ge 1$ and some $q\ge 2$, then there exists an $m$ such that $\eta_q(m)=0$ but $\eta_q(m+1)>0$, since $\eta_q(m)=1+O(1/m)$. Hence $f(n,m)\ll_{m,q} n^{-2}$ while $f(n,m+1) \asymp_{m,q} n^{-1}$.
Lemma \ref{lb} shows that this is impossible. 
Finally, the  estimate $\eta_q(m)=1 + O(m^{-1}q^{-(m+1)\tau})$ is the topic of Proposition \ref{etaqm}.
\end{proof} 

\begin{lemma}\label{lb}
For $n\ge 0$, $m\ge 1$ we have $f(n,m+1)\le q f(n+1,m)$. 
\end{lemma}

\begin{proof}
For each $F\in \mathcal{A}(n,m+1)$, we have $xF(x)\in  \mathcal{A}(n+1,m)$. Hence $| \mathcal{A}(n,m+1)| \le | \mathcal{A}(n+1,m)|$, that is $q^n f(n,m+1) \le q^{n+1} f(n+1,m)$.
\end{proof}

\section{Proof of Theorem \ref{thm5}}

Let $\mathcal{B}(n,m)$ denote the set of permutations $\sigma \in S_n$ with the property that
$A_3$ has no gaps of size greater than $m$.  We have $g(n,m)=|\mathcal{B}(n,m)|/n! $. The following five lemmas correspond to Lemmas \ref{cond} through \ref{inteq} in the last section. We omit some of the proofs of these lemmas and the first half of the proof of Theorem \ref{thm5}, since they are almost identical to those of the previous section.

\begin{lemma}\label{condp}
A permutations $\sigma \in S_n$, which decomposes into $j$ cycles with cycle lengths $l_1 \le l_2 \le \ldots \le l_j$,
satisfies $\sigma \in \mathcal{B}(n,m)$ if and only if 
\begin{equation*}
l_i \le m+ \sum_{1\le k < i} l_k  \qquad (1\le i \le j).
\end{equation*}
\end{lemma}

\begin{proof}
Follow the proof of Lemma \ref{cond}, replacing $\deg(P_i)$ by $l_i$.
\end{proof}

\begin{lemma}\label{ge}
For $n\ge 0$, $m\ge 0$ we have
$$ 1=\sum_{0\le k \le n} g(k,m) \, p(n-k,k+m)  .$$
\end{lemma}

\begin{proof}
Lemma \ref{condp} implies that each  $\sigma \in S_n$ decomposes uniquely as $\sigma=\rho \tau$, where $\sigma=\sigma_1 \sigma_2 \cdots \sigma_j$,  the $\sigma_i$ are cycles with lengths $l_1 \le l_2 \le \ldots \le l_j$, 
$$ \rho=\sigma_1\cdots \sigma_{i_0} , \quad \tau=\sigma_{i_0+1}\cdots \sigma_j ,$$
and $i_0$ is the unique index such that
$$ l_i \le  m+ \sum_{1\le r < i} l_r  \qquad (1\le i \le i_0) , \quad l_{i_0+1}>l_1+\ldots + l_{i_0}+m.$$
The case $i_0=0$ corresponds to $\tau=\sigma$, while $i_0=j$ means $\rho=\sigma \in \mathcal{B}(n,m)$.
Given that $\sum_{1\le r \le  i_0} l_r =k$, there are $\binom{n}{k}$ ways to choose the $k$ numbers from $\{1,2,\ldots,n\}$, which $\rho$ is acting on. Once these $k$ numbers are chosen, there are $k! \, g(k,m)$ choices for $\rho$ and $(n-k)! \,  p(n-k,k+m)$ choices for $\tau$. Counting all of the $n!$ permutations in $S_n$ according to $k=\sum_{1\le r \le  i_0} l_r $ , we find that
$$n! = \sum_{0\le k \le n}\binom{n}{k} k! \, g(k,m) \, (n-k)! \, p(n-k,k+m),$$
from which the result follows.
\end{proof}

\begin{lemma}\label{ge2}
For $n\ge 0$, $m\ge 0$ we have
$$ 1= g(n,m)+\sum_{0\le k < \frac{n-m}{2}} g(k,m) \, p(n-k,k+m) .$$
\end{lemma}

\begin{proof}
This follows from Lemma \ref{ge} since $p(0,n+m)=1$ and $p(n-k,k+m)=0$ if $\frac{n-m}{2} \le k < n$. 
\end{proof}

\begin{lemma}\label{sep}
For $m\ge 0$ we have
$$ 1=\sum_{k\ge 0} g(k,m) \, e^{-H_{k+m}} .$$
\end{lemma}

\begin{proof}
Follow the proof of Lemma \ref{se}.
\end{proof}

\begin{lemma}\label{inteqp}
For $n\ge 0$, $m \ge 1$ we have
$$ g(n,m) = \int_0^\infty \frac{g(t,m)}{t+m} \left( e^{-\gamma} - \omega\left(\frac{n+m}{t+m}-1\right) \right)   \mathrm{d}t + \mathcal{E}(n,m), $$
where
$$ \mathcal{E}(n,m) \ll  \mathcal{E}_0(n,m):=\int_0^\infty \frac{g(t,m)}{(t+m)^2}\exp\left(-\frac{n+m}{t+m}\right)   \mathrm{d}t  + \frac{g(\frac{n-m}{2},m)}{n+m}. $$
\end{lemma}

\begin{proof}
Follow the proof of Lemma \ref{inteq} and use Proposition \ref{fullp} (in place of Theorem \ref{thm2}) to approximate $p(n-k,k+m)$.
\end{proof}

\begin{proof}[Proof of Theorem \ref{thm5}]
As in the proof of Theorem \ref{thm3}, we arrive at the analogue of \eqref{etadef} and \eqref{lastf}, namely
\begin{equation}\label{lastg}
g(x,m)= \eta(m) \, d(x/m) +O\left(\frac{m}{(x+m)^2}\right),
\end{equation}
where $\eta(m)=1+O(1/m)$. It remains to show that $\eta(m)=1$ for all $m\ge 1$. 
To this end, we multiply the equation in Lemma \ref{ge} by $z^n$, $|z|<1$, to get
$$ z^n=\sum_{0\le k \le n} g(k,m)z^k \, p(n-k,k+m) z^{n-k} .$$
Summing over $n\ge 0$ yields
$$ \frac{1}{1-z} = \sum_{k\ge 0}  g(k,m)\, z^k \, \frac{1}{1-z} \exp \left(-\sum_{j=1}^{k+m} \frac{z^j}{j}\right),$$
by \eqref{Dgen}, and hence
$$ z^m = \sum_{k\ge m}  g(k-m,m)\,  z^k  \exp \left(-\sum_{j=1}^{k} \frac{z^j}{j}\right).$$
Differentiating both sides with respect to $z$, we obtain
\begin{equation}\label{mz}
m z^{m-1} =  \sum_{k\ge m}  g(k-m,m) \,  z^k \exp \left(-\sum_{j=1}^{k} \frac{z^j}{j}\right) \left(\frac{k}{z}-\frac{1-z^k}{1-z}\right).
\end{equation}
Our plan is to learn about $\eta(m)$ by inserting the estimate \eqref{lastg} into \eqref{mz}. 
Let $m\ge 1$ be arbitrary but fixed and let $z=1-\delta$, where $0<\delta <1/2$. 
We will need the following estimates.
First,
$$ \frac{k}{z}-\frac{1-z^k}{1-z} =\frac{k}{1-\delta} -\frac{1-\exp(-k\delta(1+O(\delta)))}{\delta}\ll 
\begin{cases}
   k^2 \delta & (k\le \delta^{-1})  \\
   k       & (k> \delta^{-1}).
  \end{cases}
$$
Second, as in \eqref{sumzj}, we have
\begin{equation*}
\begin{split}
\sum_{j=1}^{k} \frac{z^j}{j} 
& = H_k+ \int_{\log z}^0 (1-e^{kt})  \, \mathrm{d}t - \int_{\log z}^0  \frac{1-e^{kt}}{1-e^t}  \, \mathrm{d}t \\
& \ge H_k + 0 -  \int_{\log z}^0  \frac{1-e^{kt}}{-t} (1+O(t))  \, \mathrm{d}t \\
& \ge \log k -\max(0,\log(-k\log z)) -O(1) \\
&\ge \min(\log \delta^{-1}, \log k) - O(1),
\end{split}
\end{equation*}
hence
$$  \exp \left(-\sum_{j=1}^{k} \frac{z^j}{j}\right) \ll \max(\delta, 1/k).$$
Since $g(k-m,m)\ll m/k$ by \eqref{lastg}, the contribution from $k\le \delta^{-1/2}$ to the right-hand side of \eqref{mz} is
$$\ll \sum_{m\le k\le \delta^{-1/2}} \frac{m}{k}\cdot 1 \cdot  \frac{1}{k}\cdot k^2 \delta \le m \delta^{1/2}. $$
For $k>\delta^{-1/2}$, we estimate $g(k-m,m)$ in \eqref{mz} by
$$g(k-m,m)= \frac{C\eta(m) m}{k} +O\left(\frac{m}{k^2}+\frac{m^3}{k^3}\right),$$
which follows from \eqref{lastg} and \eqref{IDD3T1}. The contribution from the error terms is
$$ \ll \sum_{\delta^{-1/2}<k\le \delta^{-1}} \frac{m^3}{k^2} \cdot 1 \cdot \frac{1}{k} \cdot k^2 \delta
+ \sum_{k>\delta^{-1}} \frac{m^3}{k^2} \cdot e^{-\delta k} \cdot \delta \cdot k
\ll m^3\delta \log (\delta^{-1}) .
$$
After inserting the terms corresponding to $1\le k\le \delta^{-1/2}$, which contribute $\ll m\delta^{1/2}$, we get
$$ m z^{m-1} =O\left(m^3 \delta^{1/2}\right)+  \sum_{k\ge 1} \frac{C\eta(m) m}{k} z^k \,  \exp \left(-\sum_{j=1}^{k} \frac{z^j}{j}\right) \left(\frac{k}{z}-\frac{1-z^k}{1-z}\right),$$
where $z=1-\delta$. Taking the limit as $\delta \to 0^+$, we find that
$$ \frac{1}{C\eta(m)} = \lim_{z\to 1^-} \sum_{k\ge 1} \frac{z^k }{k} \,  \exp \left(-\sum_{j=1}^{k} \frac{z^j}{j}\right) \left(\frac{k}{z}-\frac{1-z^k}{1-z}\right).$$
Since the right-hand side is independent of $m$, so is $\eta(m)$. But $\eta(m)=1+O(1/m)$. Thus $\eta(m)=1$ for all $m\ge 1$ and the proof of Theorem \ref{thm5} is complete.
\end{proof}

\section{An estimate for $\eta_q(m)$.}

Define $\kappa $ to be the unique positive constant satisfying 
\begin{equation}\label{kappa}
1=\int_1^\infty \omega(y)\, (y+1)^{-1-\kappa}\,  \mathrm{d} y  .
\end{equation} 
Numerical calculations show that $\kappa = 0.433489...$.
This estimate was obtained using exact formulas for $\omega(y)$ for $0\le y \le 5$, derived with the help of Mathematica; we used a table of zeros and relative extrema of $\omega(u)-e^{-\gamma}$ on the interval $[5, 10.3355]$ due to Cheer and Goldston \cite{CG}, and the estimate $|\omega(u)-e^{-\gamma}|<1/\Gamma(u+1)$ from \cite[Lemma 1]{IDD3} for $u\ge 10.3355$.

\begin{proposition}\label{fandg}
For $n\ge 0$, $m\ge 1$ we have
$$ f(n,m)=g(n,m) + O\left(\frac{n^\kappa}{m^{1+\kappa} q^{(m+1)/2}}\right).$$
\end{proposition}

\begin{proof}
We show the result for all $m\ge 1$ and $n\le m 2^j$ by induction on $j$. If $n\le m$, then $f(n,m)=g(n,m)=1$. 
Assume that, for some $j\ge 0$, there is a constant $B_j\ge 1$, such that for all $m\ge 1$ and $n\le m 2^j$,
$$|f(n,m)  -g(n,m)|\le \frac{B_j (n+m)^\kappa}{m^{1+\kappa} q^{(m+1)/2}}.$$ 
Now let $m\ge 1$ and $m2^j < n \le m 2^{j+1}$.  
From Lemmas \ref{fe2} and \ref{ge2} we have 
\begin{align*}
|f(n,m) & -g(n,m)| \\ 
& \le \sum_{0\le k < \frac{n-m}{2}}\Bigl| g(k,m) \, p(n-k,k+m) - f(k,m) \, r(n-k,k+m)\Bigr| \\
 & \le S_1 + S_2,
\end{align*}
where, by Proposition \ref{pub},
\begin{align*}
S_1 & := \sum_{0\le  k < \frac{n-m}{2}}p(n-k,k+m) \Bigl| g(k,m) - f(k,m)\Bigr|\\
& \le  \sum_{0\le  k < \frac{n-m}{2}} \left( \frac{\omega\left(\frac{n-k}{k+m}\right)}{k+m} + \frac{O(1)}{(k+m)^2}\right) 
\frac{B_j (k+m)^\kappa}{m^{1+\kappa} q^{(m+1)/2}} \\
& \le \frac{B_j }{m^{1+\kappa} q^{(m+1)/2}}
\left( \int_0^\frac{n-m}{2}  \frac{\omega\left(\frac{n-t}{t+m}\right)}{(t+m)^{1-\kappa}} \, \mathrm{d}t 
+ O\left(\frac{1}{m^{1-\kappa}}\right)\right) \\
& \le  \frac{B_j (n+m)^\kappa}{m^{1+\kappa} q^{(m+1)/2}}
\left( \int_1^\infty  \omega\left(y\right) (y+1)^{-1-\kappa} \, \mathrm{d}y 
+ O\left(\frac{m^\kappa}{m (n+m)^{\kappa}}\right)\right) \\
& =  \frac{B_j (n+m)^\kappa}{m^{1+\kappa} q^{(m+1)/2}}
\left( 1+ O\left(\frac{1}{m 2^{j \kappa}}\right)\right).
\end{align*}
 By  Proposition \ref{rap},
\begin{align*}
S_2 & := \sum_{0\le k < \frac{n-m}{2}}f(k,m) \Bigl| p(n-k,k+m) - r(n-k,k+m)\Bigr|\\
& \ll  \sum_{0\le k < \frac{n-m}{2}} \frac{1 }{(k+m) q^{(k+m+1)/2}}  \ll \frac{1}{m q^{(m+1)/2}} \\
& \le  \frac{ (n+m)^\kappa}{m^{1+\kappa} q^{(m+1)/2}} \frac{1}{2^{j \kappa}}.
\end{align*}
Combining these estimates, we obtain
$$ |f(n,m)  -g(n,m)|\le  \frac{B_{j+1} (n+m)^\kappa}{m^{1+\kappa} q^{(m+1)/2}}  \qquad (m\ge 1, \  n\le m 2^{j+1}),
$$
where $B_{j+1}:=B_j(1+A 2^{-j \kappa})$ and $A$ is a suitable absolute constant.
\end{proof}

\begin{proposition}\label{etaqm}
Let $\tau = 1/(4+2\kappa) =0.205466...$, where $\kappa$ is given by \eqref{kappa}.
For $m\ge 1$, we have $$\eta_q(m) = 1 +  O\left(\frac{1}{m  q^{(m+1)\tau}}\right).$$
\end{proposition}

\begin{proof}
Assume $n\ge m$. In Proposition \ref{fandg}, replace $f(n,m)$ and $g(n,m)$ by their respective estimates \eqref{lastf} and \eqref{lastg}, with $\eta(m)=1$.
We obtain
$$ \eta_q(m) d(n/m) =  d(n/m) +O\left(\frac{m}{n^2}+ \frac{n^\kappa}{m^{1+\kappa} \, q^{(m+1)/2}}\right).$$
Dividing by $d(n/m) \asymp m/n$ gives
$$ \eta_q(m)=1+ O\left(\frac{1}{n}+ \frac{n^{1+\kappa}}{m^{2+\kappa} \, q^{(m+1)/2}}\right).$$
The result now follows with $n=\lfloor m q^{(m+1)/(4+2\kappa)} \rfloor = \lfloor m  q^{(m+1)\tau} \rfloor$.
\end{proof}

\end{document}